\documentclass[11pt]{amsart}
\usepackage{amssymb}
\usepackage{mathtools}
\usepackage{color}

\newtheorem{theo}{Theorem} 

\newtheorem{lemma}{Lemma}[section]
\newtheorem{prop}[lemma]{Proposition}

\newtheorem{claim}[lemma]{Claim}

\theoremstyle{remark}

\theoremstyle{definition}

\newcommand{\RR}{\mathbb{R}}

\newcommand{\eps}{\varepsilon}

\newcommand{\PPP}{\mathcal{P}}

\newcommand{\tZ}{\widetilde{Z}}

\newcommand{\te}{\tilde{e}}

\newcommand{\tE}{\widetilde{E}}

\newcommand{\vu}{\vec{u}}

\DeclareMathOperator{\im}{Im}

\newcommand{\hdot}{\dot{H}^1}

\newcommand{\EMPH}[1]{\medskip\noindent\textit{#1}.}

\DeclareMathOperator{\supp}{supp}

\DeclareMathOperator{\Div}{div}

\numberwithin{equation}{section} 

\title[Global existence for focusing wave equations]{Global existence for solutions of the focusing wave equation with the compactness property}
\author[T.~Duyckaerts]{Thomas Duyckaerts$^1$}
\author[C.~Kenig]{Carlos Kenig$^2$}
\author[F.~Merle]{Frank Merle$^3$}
\thanks{$^1$LAGA, Universit\'e Paris 13 (UMR 7539). Partially supported by ERC Grant Dispeq and ERC advanced grant  no. 291214, BLOWDISOL}
\thanks{$^2$University of Chicago. Partially supported by NSF Grants DMS-0968472 and DMS-1265249}
\thanks{$^3$Cergy-Pontoise (UMR 8088), IHES. Partially supported by ERC advanced grant  no. 291214, BLOWDISOL}

\date{\today}

\begin{document}
\begin{abstract}
 We prove that every solution of the focusing energy-critical wave equation with the compactness property is global. We also give similar results for supercritical wave and Schr\"odinger equations.
\end{abstract}

\maketitle

\section{Introduction}
In this note we consider solutions with the compactness property for (mainly) the energy-critical wave equation in dimension $N\in \{3,4,5\}$. This is the equation
\begin{equation}
 \label{CP}
 \left\{\begin{aligned}
         \partial_t^2u -\Delta u&=|u|^{\frac{4}{N-2}}u,\quad t\in I,\; x\in \RR^N\\
         (u,\partial_tu)_{\restriction t=0}&=\vu_0\in \dot H^1\times L^2,
        \end{aligned}\right.
\end{equation} 
where $I$ is an interval ($0\in I$), $u$ is real-valued, $\hdot=\hdot(\RR^N)$ and $L^2=L^2(\RR^N)$. For such solutions $(u,\partial_tu)\in C^0(I,\dot{H}^1\times L^2)$, we denote the maximal interval of existence $(T_-(u),T_+(u))=I_{\max}(u)$. We say that a solution has \emph{the compactness property} if there exists $\lambda(t)>0$, $x(t)\in \RR^N$, $t\in I_{\max}(u)$ such that
\begin{equation}
 \label{defK}
K=\left\{\left( \lambda^{\frac{N-2}{2}}(t)u\big( t,\lambda(t)x+x(t)\big), \lambda^{\frac{N}{2}}(t)\partial_tu\big( t,\lambda(t)x+x(t) \big)\right)\;:\; t\in I_{\max}(\vec{u}_0)\right\}
 \end{equation} 
is precompact in $\dot{H}^1\times L^2$. Solutions with the compactness property have been extensively studied (see for instance \cite{KeMe08,DuKeMe11a,DuKeMe12,DuKeMe13Pb} for equation \eqref{CP}). The reason for these studies is that, if one considers solutions to \eqref{CP} such that 
$$\sup_{0<t<T_+(u)} \|(u(t),\partial_tu(t))\|_{\dot{H}^1\times L^2}<\infty$$
and which do not scatter, there always exist $t_n\to T_+(u)$ such that, up to modulation, $(u,\partial_tu)(t_n)$ weakly converges to 
$(U(0),\partial_tU(0))$ in $\hdot\times L^2$, where $U$ is a solution with the compactness property. This is Proposition 1.10 in \cite{DuKeMe15a}, which, as is noted there, is valid for a wide class of dispersive equations (see also \cite{Tao07DPDE} for nonlinear Schr\"odinger equation in high space dimension). This clearly shows the crucial role played by solutions with the compactness property in the study of bounded solutions of dispersive equations, with no a priori size restriction. Note also that in the case of \eqref{CP}, a much more precise result (Theorem 1 in \cite{DuKeMe15a}) is valid.

In \cite{DuKeMe11a} we showed that, in the radial case, up to scaling and sign change, the only solution of \eqref{CP} with the compactness property is the ``ground-state'' $W(x)=\left( \frac{1}{1+|x|^2/N(N-2)} \right)^{\frac{N-2}{2}}$, which is the only non-zero radial solution in $\dot{H}^1(\RR^N)$ (up to scaling and sign change) of the elliptic equation $\Delta u+|u|^{\frac{4}{N-2}}u=0$. In the sequel, we will denote by $\Sigma$ the set of non-zero solutions to this elliptic equation. In \cite{DuKeMe15a}, Proposition 1.8, a), we showed (by a simple virial argument), that if $u$ has the compactness property, then $T_-(u)=-\infty$, or $T_+(u)=+\infty$. In \cite{KeMe08}, the second and third authors showed that if $T_+(u)<\infty$ and $u$ has the compactness property, then there exists $x_+\in \RR^N$ such that if $t\in I_{\max}$, 
$$\supp (u(t),\partial_tu(t))\subset \{ |x-x_+|\leq |T_+(u)-t|\}$$
and 
$$\lim_{t\to T_+(u)} \frac{\lambda(t)}{T_+(u)-t}=0.$$
In particular, the self similar blow-up, given by $\lambda(t)\approx T_+(u)-t$, is excluded. 

In \cite{DuKeMe15a}, Proposition 1.8 b), we showed that, if $u$ has the compactness property, then there exist two sequences $\{t_n^{\pm}\}$ in $(T_-(u),T_+(u))$, with $\lim_{n\to \pm\infty}=T_{\pm}(u)$ and two elements $Q^{\pm}$ of $\Sigma$ and a vector $\vec{\ell}$ with $|\vec{\ell}|<1$, such that, up to modulation, $(u(t_n^{\pm}),\partial_tu(t_n^{\pm}))\to (Q_{\vec{\ell}}^{\pm}(0),\partial_tQ_{\vec{\ell}}^{\pm}(0))$ strongly in $\hdot\times L^2$, where $Q^{\pm}_{\vec{\ell}}$ is the Lorentz transform of the solution $Q^{\pm}$, given by 
$$ Q^{\pm}_{\vec{\ell}}=Q^{\pm}\left(\left(-\frac{t}{\sqrt{1-|\vec{\ell}|^2}}+\frac{1}{|\vec{\ell}|^2} \left(\frac{1}{\sqrt{1-|\vec{\ell}|^2}}-1\right)\vec{\ell}\cdot x\right)\vec{\ell}+x\right),$$
so that $Q_{\vec{\ell}}^{\pm}(t,x)=Q^{\pm}_{\vec{\ell}}(0,x-t\vec{\ell})$, which are clearly solutions of \eqref{CP} with the compactness property. In \cite{DuKeMe13Pb}, we proved that the class of solutions with the compactness property is invariant under Lorentz transformation. In light of this result, we have (see \cite{DuKeMe15a,DuKeMe13Pb}) the rigidity conjecture for solutions with the compactness property: $0$ and $Q_{\vec{\ell}}$, $Q\in \Sigma$, $|\vec{\ell}|<1$, are the only solutions of \eqref{CP} with the compactness property. In \cite{DuKeMe13Pb} we proved this conjecture, under a nondegeneracy assumption on $Q^+$ (or $Q^-$). 

The main result in this note is an extension of the non-existence of self-similar solution with the compactness property in \cite{KeMe08}. Namely:
\begin{theo}
\label{T:main}
 Let $u$ be a solution of \eqref{CP} with the compactness property. Then $u$ is global.
\end{theo}
Theorem \ref{T:main} is proved in Section \ref{S:main}.
We will also prove a similar result for supercritical nonlinear wave equation (see Section \ref{S:super}) and supercritical Schr\"odinger equation (see Section \ref{S:NLS}).

In all the article, we let $\vu=(u,\partial_tu)$.

\section{Energy-critical case}
\label{S:main}
In this section we prove Theorem \ref{T:main}. We recall that the energy:
\begin{equation}
\label{defE}
E(\vu)=\frac 12\int_{\RR^N} |\nabla u(t,x)|^2\,dx+\frac{1}{2}\int_{\RR^N} (\partial_tu(t,x))^2\,dx-\frac{N-2}{2N}\int_{\RR^N} |u(t,x)|^{\frac{2N}{N-2}}\,dx 
\end{equation} 
is well-defined and independent of $t$. 
The proof relies on the following two propositions, which we will prove in Subsection \ref{SS:SSvariables} using a self-similar change of variables.
\begin{prop}
 \label{P:critical}
 Let $u$ be a solution of \eqref{CP} with $p=\frac{N+2}{N-2}$ such that 
 \begin{gather}
 \label{C2}
 T_+(\vec{u}_0)=1,\quad T_-(\vec{u}_0)=-\infty.\\
 \label{C3}
 \supp u\subset\left\{ |x|\leq 1-t\right\}\\
 \label{C4}
 \sup_{t<1}\|\nabla u(t)\|^2_{L^2}+\|\partial_t u(t)\|^2_{L^2}=M_0<\infty.
 \end{gather}
Then
\begin{equation}
\label{C5}
\limsup_{T\to-\infty}\frac{1}{\log(2-T)} \frac{1}{N}\int_T^0\frac{1}{2-t}\int_{\RR^N}|u(t,x)|^{\frac{2N}{N-2}}\,dx\,dt<E(\vec{u}_0). 
\end{equation} 
\end{prop}
In the next proposition, we denote by $P(\vec{u})$ the momentum
$$P(\vec{u}(t))=\int_{\RR^N}\nabla u\partial_tu\,dx \in \RR^N,$$
which is independent of time for a solution $u$ of \eqref{CP}.

\begin{prop}
 \label{P:lambda}
 Let $u$ be a solution with the compactness property such that $P(\vu_0)=0$ and $T_+(u)<\infty$, $T_-(u)=-\infty$. Then 
$$ \lim_{t\to-\infty} \frac{\lambda(t)}{t}=0.$$
\end{prop}
\subsection{Proof of the main result}
We first assume Propositions \ref{P:critical} and \ref{P:lambda} and prove Theorem \ref{T:main}. We argue by contradiction. Let $u$ be a solution of \eqref{CP} such that $T_+(u)<\infty$. Translating in time, we can assume $T_+(u)=1$. By \cite{DuKeMe15a}, $T_-(u)=-\infty$. 
As in the beginning of Subsection 4.1 in \cite{DuKeMe13Pb}, we can use the Lorentz transform on $u$ and assume that $P(\vu_0)=0$. Using a by now standard argument (see \cite{KeMe08}), this implies
\begin{equation}
 \label{xt}
\lim_{t\to-\infty}\frac{x(t)}{t}=0.
\end{equation} 
Let
\begin{equation*}
Z(t)=-\int_{\RR^N} x\cdot\nabla u\partial_t u-\frac{N-2}{2}\int_{\RR^N} u\partial_t u,\quad
\tZ(t)=-\int_{\RR^N} x\cdot\nabla u\partial_t u-\frac{N}{2}\int_{\RR^N} u\partial_t u.
\end{equation*}
Then
\begin{equation}
 \label{monotonicity}
Z'(t)=\int_{\RR^N} (\partial_tu)^2,\quad \tZ'(t)=\int_{\RR^N} \left(|\nabla u|^2-|u|^{\frac{2N}{N-2}}\right).
\end{equation} 
Using Proposition \ref{P:lambda}, \eqref{xt} and the precompactness of $K$, one gets
$$\lim_{t\to-\infty}\frac{1}{t}Z(t)=\lim_{t\to-\infty}\frac{1}{t}\tZ(t)=0.$$
Integrating \eqref{monotonicity}, we deduce:
\begin{gather}
 \label{C6}
 \lim_{t\to-\infty}\frac{1}{t}\int_{t}^0\int_{\RR^N}(\partial_tu)^2\,dx\,dt=0\\
 \label{C7}
 \lim_{t\to-\infty}\frac{1}{t}\int_{t}^0\int_{\RR^N}|\nabla u(t,x)|^2-|u(t,x)|^{\frac{2N}{N-2}}\,dx\,dt=0.
 \end{gather} 
Combining \eqref{C6} and \eqref{C7} with the identity $$\frac{1}{N}\int_{\RR^N} |u|^{\frac{2N}{N-2}}=E(\vu_0)-\frac{1}{2}\left( \int_{\RR^N} |\nabla u|^2-\int_{\RR^N} |u|^{\frac{2N}{N-2}} \right)-\frac 12 \int_{\RR^N} (\partial_t u)^2,$$
we obtain 
\begin{equation}
 \label{C9}
 \lim_{t\to -\infty} \frac{1}{N\,t}\int_{t}^0\int_{\RR^N} |u(t,x)|^{\frac{2N}{N-2}}\,dx\,dt=E(\vu_0).
\end{equation} 
By \cite{KeMe08}, we have $\supp u\subset \{|x|\leq 1-t\}$. Thus $u$ satisfies the assumptions of Proposition \ref{P:critical},
and \eqref{C9} contradicts the conclusion \eqref{C5} of Proposition \ref{P:critical}, in view of the following elementary claim, proved in the appendix:
\begin{claim}
 Let $g:[0,+\infty)\to \RR$, bounded and continuous, such that 
 \begin{equation}
  \label{C10}
  \lim_{T\to+\infty} \frac{1}{T}\int_0^T g(t)\,dt=\ell\in \RR.
 \end{equation} 
 Then:
 \begin{equation}
  \label{C11}
  \lim_{T\to+\infty}
  \frac{1}{\log(2+T)}\int_0^T \frac{g(t)}{2+t}\,dt=\ell.
 \end{equation} 
\end{claim}
\subsection{Self-similar variables}
\label{SS:SSvariables}
We complete here the proof of Theorem \ref{T:main} by proving Propositions \ref{P:critical} and \ref{P:lambda}. As in \cite[Section 6]{KeMe08}, we use a self-similar change of variables and a Lyapunov functional in the new variables (see \cite{GiKo89} for the introduction of this change of variable for heat equations and \cite{MeZa03} for semilinear wave equations). The main novelty of the proofs is that we also use this change of variables for negative times (and indeed, as $t\to-\infty$), whereas it was only applied close to the blow-up time in \cite{KeMe08}. 

We start with a few preliminary lemmas that are essentially contained in \cite[Section 6]{KeMe08}.
 Let $u$ be as in Proposition \ref{P:critical}. 
Let:
\begin{gather}
 \label{C13}
 y=\frac{x}{2-t},\quad s=-\log(2-t)\\
 \label{C14} 
 w(s,y)=(2-t)^{\frac{N-2}{2}}u(t,x)=e^{-\frac{s(N-2)}{2}}u(2-e^{-s},e^{-s}y)
\end{gather}
(these are exactly the changes of variables and unknown functions in \cite{KeMe08} with $\delta=1$).

Then $w(s,y)$ is defined for $s<0$ and 
\begin{equation}
 \label{C15}
 \supp w\subset \Big\{ |y|\leq 1-e^s\leq 1\Big\}.
\end{equation} 
Furthermore $w$ solves, for $s<0$, the equation
\begin{equation}
 \label{C16}
 \partial_s^2w=\frac{1}{\rho} \Div\left(\rho\nabla w-\rho(y\cdot\nabla w)y\right) -\frac{N(N-2)}{4}w+|w|^{\frac{4}{N-2}}w-2y\cdot \nabla\partial_s w-(N-1)\partial_s w,
\end{equation} 
where $\rho=(1-|y|^2)^{-1/2}$. We start with a few lemmas, that are essentially contained in \cite[Section 6]{KeMe08}. In all the proof, $C$ is a large positive constant that may change from line to line and depends on the constant $M_0$ defined in \eqref{C4}. We denote by $B_1$ the unit ball of $\RR^N$. 
\begin{lemma}
 \label{L:C3}
 For $s<0$, $t=2-e^{-s}$, we have
 \begin{gather}
  \label{C17}
  \int_{B_1}|w(s,y)|^{\frac{2N}{N-2}}\,dy=\int_{\RR^N}|u(t,x)|^{\frac{2N}{N-2}}\,dx,\quad \int_{B_1} |\nabla_y w(s,y)|^2\,dy=\int_{\RR^N}|\nabla_xu(t,x)|^2\,dx\\
  \label{C18}
  \int_{B_1} |w(s,y)|^2\frac{dy}{(1-|y|^2)^2}\leq C,\quad \int_{B_1}|\partial_sw(s,y)|^2\,dy\leq C\\
  \label{C19}
  \int_{B_1}\left(|\nabla_y w(s,y)|^2+|\partial_sw(s,y)|^2+|w(s,y)|^{\frac{2N}{N-2}}+|w(s,y)|^2\right)\frac{dy}{(1-|y|^2)^{1/2}}\leq Ce^{|s|/2}\\
  \label{C20}
  \int_{B_1}\left(|\nabla_y w(s,y)|^2+|\partial_sw(s,y)|^2+|w(s,y)|^{\frac{2N}{N-2}}+|w(s,y)|^2\right)\log\frac{1}{(1-|y|^2)}\,dy\leq C|s|.
 \end{gather}
\end{lemma}
\begin{proof}
 The identities \eqref{C17} follow directly from the change of variable $y=\frac{x}{2-t}$.
 
 The first inequality in \eqref{C18} is a consequence of the condition \eqref{C15} on the support of $w$, of \eqref{C17} and of Hardy's inequality (see e.g. \cite{BrMa97}).
 
 To obtain the second bound in \eqref{C18}, we write
 \begin{equation}
  \label{C20'}
  \partial_sw=-\frac{N-2}{2}w-y\cdot\nabla w+e^{-\frac N2s}\partial_tu(2-e^{-s},e^{-s}y),
 \end{equation} 
 and the desired bound follows from \eqref{C17}, the first inequality in \eqref{C18} and the identity
 $$\int\left(e^{-\frac N2s}\partial_tu(2-e^{-s},e^{-s}y)\right)^2\,dy=\int (\partial_tu(t,x))^2\,dx.$$ 

 The estimates \eqref{C19} and \eqref{C20} follow from \eqref{C17} and \eqref{C18} and the fact that, on the support of $w$,
 $1-|y|^2\geq 1-|y|\geq e^{s}$.
 \end{proof}
\begin{lemma}
 \label{L:C4}
 Let for $s<0$, $y\in B_1$,
 $$\te(s,y)=\frac{1}{2}\left((\partial_sw)^2+|\nabla w|^2-(y\cdot \nabla w)^2\right)+\frac{N(N-2)}{8}w^2-\frac{N-2}{2N}|w|^{\frac{2N}{N-2}},$$
 and 
 $$\tE(s)=\int_{B_1}\frac{\te(s,y)}{(1-|y|^2)^{1/2}}\,dy.$$
 Then, if $s_1<s_2<0$,
 \begin{equation}
  \label{C21}
  \tE(s_2)-\tE(s_1)=\int_{s_1}^{s_2} \int_{B_1} \frac{(\partial_sw)^2}{(1-|y|^2)^{3/2}}\,dy\,ds
  \end{equation}
  \begin{multline}
  \label{C22}
  \frac{1}{2}\int_{B_1}\Big(\partial_sw\,w-\frac{1+N}{2}w^2\Big)\frac{dy}{(1-|y|^2)^{1/2}}\Bigg|_{s_1}^{s_2}\\
  =
  -\int_{s_1}^{s_2} \tE(s)\,ds+\frac 1N\int_{s_1}^{s_2} \int_{B_1} \frac{|w|^{\frac{2N}{N-2}}}{(1-|y|^2)^{1/2}}\,dy\,ds\\
  \quad +\int_{s_1}^{s_2} \int_{B_1} \left((\partial_s w)^2+\partial_sw y\cdot\nabla w+\frac{|y|^2\partial_s w\,w}{1-|y|^2}\right)\frac{dy}{(1-|y|^2)^{1/2}}\,ds
  \end{multline}
  \begin{multline}
  \label{C23}
  \int_{B_1} \te(s,y)(-\log(1-|y|^2))\,dy\Bigg|_{s_1}^{s_2} +\int_{s_1}^{s_2}\int_{B_1} \left(2+\log(1-|y|^2)\right)y\cdot\nabla w\partial_sw\,dy\,ds\\-\int_{s_1}^{s_2} \int_{B_1} \log(1-|y|^2)(\partial_sw)^2\,dy\,ds
  -2\int_{s_1}^{s_2} \int_{B_1} (\partial_sw)^2\,dy\,ds=-2\int_{s_1}^{s_2}\int_{B_1}\frac{(\partial_sw)^2}{1-|y|^2}\,dy\,ds.
 \end{multline}
\end{lemma}

The proof is by direct computations. See \cite{MeZa03} for \eqref{C21} and \eqref{C22}. The identity \eqref{C23} is the first identity in the proof of Lemma 6.4 in \cite{KeMe08}.
\begin{lemma}
 \label{L:C5}
 $$\lim_{s\overset{<}{\longrightarrow} 0}\tE(s)=E(\vu_0).$$
\end{lemma}
\begin{proof}
This is contained in Proposition 6.2 (iii) of \cite{KeMe08}. We give the proof for completeness.
Since $|y|\leq 1-e^s$ on the support of $w$, we have, by \eqref{C17} and Hardy's inequality $\int \frac{1}{|y|^2}|f|^2\leq \int|\nabla f|^2$,
\begin{equation}
 \label{C24}
 \lim_{s\overset{<}{\longrightarrow} 0} \int (y\cdot \nabla w)^2\frac{dy}{(1-|y|^2)^{1/2}}+\int |w|^2\frac{dy}{(1-|y|^2)^{1/2}}=0
\end{equation} 
and 
\begin{equation}
 \label{C25}
 \lim_{s\overset{<}{\longrightarrow} 0} \int \left|\frac{|\nabla w|^2}{(1-|y|^2)^{1/2}}-|\nabla w|^2\right|\,dy=
 \lim_{s\overset{<}{\longrightarrow} 0} \int \left|\frac{|w|^{\frac{2N}{N-2}}}{(1-|y|^2)^{1/2}}-|w|^{\frac{2N}{N-2}}\right|\,dy=0,
 \end{equation} 
 which implies
 \begin{multline}
  \label{C26}
   \lim_{s\overset{<}{\longrightarrow} 0} \Bigg[\int \left(\frac{|\nabla w|^2}{2}-\frac{N-2}{2N} |w|^{\frac{2N}{N-2}}\right)\frac{dy}{(1-|y|^2)^{1/2}}\\
   -\int \left(\frac{|\nabla u(2-e^{-s},x)|^2}{2}-\frac{N-2}{2N} |u(2-e^{-s},x)|^{\frac{2N}{N-2}}\right)\,dx    \Bigg]=0.
 \end{multline} 
Furthermore (using \eqref{C20'} and \eqref{C24}),
\begin{equation}
 \label{C27} 
   \lim_{s\overset{<}{\longrightarrow} 0} \Bigg[\int \frac{(\partial_sw)^2}{2}\frac{dy}{(1-|y|^2)^{1/2}}-\int \frac{(\partial_tu(2-e^{-s},x))^2}{2}\,dx\Bigg]=0.
\end{equation} 
Combining \eqref{C24}, \eqref{C26} and \eqref{C27} we obtain the conclusion of the lemma.
 \end{proof}
 The next lemma is the analog of Lemma 6.4 of \cite{KeMe08}:
 \begin{lemma}
 \label{L:C6}
 $$\forall \sigma<0,\quad \int_{\sigma}^0\int_{B_1}\frac{(\partial_sw)^2}{1-|y|^2}\,dy\,ds\leq C|\sigma|.$$
\end{lemma}
\begin{proof}
 We use the identity \eqref{C23} with $s_1=\sigma$, $s_2\overset{<}{\longrightarrow} 0$.
 
In the left-hand side of \eqref{C23}, the terms with $\tilde{e}$ are bounded by $C|\sigma|$ according to \eqref{C20} in Lemma \ref{L:C3}. The term with $\log(1-|y|^2)(\partial_sw)^2$ has the good sign and can be ignored. The term with $\iint (\partial_sw)^2$ is bounded by $C|\sigma|$ according to \eqref{C18}. Furthermore,
\begin{multline*}
 \left|\int_{\sigma}^0 \int_{B_1}(2+\log(1-|y|^2))y\cdot\nabla w\partial_s w\,dy\,d\sigma\right|\\
 \leq \sqrt{\int_{\sigma}^0 \int_{B_1} (\partial_sw)^2\frac{dy}{1-|y|^2}ds}\sqrt{\int_{\sigma}^0 \int_{B_1} (1-|y|^2)\left(2+\log(1-|y|^2)\right)^2\,|\nabla w|^2\,dy\,ds}\\
 \leq C\sqrt{\int_{\sigma}^0\int_{B_1}(\partial_sw)^2\frac{dy}{1-|y|^2}\,d\sigma}\sqrt{|\sigma|}, 
\end{multline*}
by \eqref{C17}. This term can be absorbed by the inequality $ab\leq \frac{\eps}{2}a^2+\frac{1}{2\eps}b^2$.
\end{proof}

\begin{proof}[Proof of Proposition \ref{P:critical}]
The proof is divided in 4 steps.

\EMPH{Step 1} We prove  (see Lemma 6.5 of \cite{KeMe08})
\begin{equation}
 \label{C28}
 \forall \sigma<0,\quad \int_{\sigma}^0 \int_{B_1} \frac{|w(\sigma,y)|^{\frac{2N}{N-2}}}{(1-|y|^2)^{1/2}}\,dy\,ds\leq C|\sigma|.
\end{equation} 
We use the identity \eqref{C22} with $s_1=\sigma$, $s_2\overset{<}{\longrightarrow}0$. We have, for $j=1,2$,
\begin{equation*}
 \left|\int_{B_1} (\partial_s w\,w)(s_j,y)\frac{dy}{(1-|y|^2)^{1/2}}\right|\leq \sqrt{\int_{B_1}(\partial_sw)^2(s_j,y)\,dy}\sqrt{\int_{B_1}|w|^2(s_j,y)\frac{dy}{1-|y|^2}}\leq C
\end{equation*}
by \eqref{C18}. Using again \eqref{C18},
\begin{equation*}
 \int_{B_1}w^2(s_j,y)\frac{dy}{(1-|y|^2)^{1/2}}\leq C.
\end{equation*}
By Lemma \ref{L:C5} and since $\tE$ is nondecreasing: 
\begin{equation*}
 \int_{\sigma}^0\tE(s)\,ds\leq |\sigma|E(\vu_0)
\end{equation*}
In the last line of \eqref{C22}, the first term is $\geq 0$ and can be dropped. Moreover,
$$\left|\int_{\sigma}^0\int_{B_1} \partial_swy\cdot\nabla w\frac{dy}{(1-|y|^2)^{1/2}}ds\right|\leq C|\sigma|$$
by Cauchy-Schwarz, \eqref{C17} and Lemma \ref{L:C6}. Also
$$\left|\int_{\sigma}^0 \int_{B_1}\frac{|y|^2\partial_s w\,w}{(1-|y|^2)^{1/2}}dyds\right|\leq C|\sigma|,$$
by Cauchy-Schwarz and \eqref{C18}. Combining the preceding estimates, we obtain the conclusion \eqref{C28} of Step 1.

\EMPH{Step 2} We prove
\begin{gather}
\label{C29}
 \forall \sigma<0,\quad -C\leq \tE(\sigma)\leq E(\vu_0)\\
 \label{C30}
 \int_{-\infty}^0 \int_{B_1} \frac{(\partial_sw)^2}{(1-|y|^2)^{3/2}}dy\,ds<\infty.
\end{gather}
The inequality \eqref{C30} follows immediately from \eqref{C29} and the identity \eqref{C21}.

The bound from above in \eqref{C29} follows again from \eqref{C21} and Lemma \ref{L:C5}. 

By Step 1, there exists a sequence $s_n\to-\infty$ such that 
$$\forall n,\quad \int_{B_1} \frac{|w(s_n,y)|^{\frac{2N}{N-2}}}{(1-|y|^2)^{1/2}}\,dy\leq C,$$
where the constant $C$ is the same than in \eqref{C28}. Using the definition of $\tE$, we deduce 
$$\forall n,\quad \tE(s_n)\geq -C\frac{N-2}{2N}$$
and the conclusion follows using the monotonicity of the energy.

\EMPH{Step 3} We prove
\begin{equation}
 \label{C31}
 \forall \sigma <0,\quad \int_{\sigma}^0 \left(\int_{B_1}\frac{|w(s,y)|^{\frac{2N}{N-2}}}{N(1-|y|^2)^{1/2}}\,dy-\tE(s)\right)\,ds\leq C|\sigma|^{1/2}.
\end{equation} 
If $-1<\sigma <0$, \eqref{C31} follows immediately from \eqref{C28} and the boundedness of $\tE$.

We next assume $\sigma\leq -1$. We use again \eqref{C22} with $s_1=\sigma$, $s_2\overset{<}{\longrightarrow}0$.

The terms on the first line of \eqref{C22} are bounded, according to \eqref{C18} in Lemma \ref{L:C3}.

The second line of \eqref{C22} is exactly the left-hand side of \eqref{C31}. We are left with bounding the third line of \eqref{C22}. We have:
$$\int_{\sigma}^0 \int_{B_1} \frac{(\partial_sw)^2}{(1-|y|^2)^{1/2}}dyds\leq C$$
by Step 2.
\begin{equation*}
\left|\int_{\sigma}^0\int_{B_1} \frac{\partial_swy\cdot\nabla w}{(1-|y|^2)^{1/2}}dyds\right| \leq \sqrt{\int_{\sigma}^0 \int_{B_1} \frac{(\partial_sw)^2}{1-|y|^2}dyds}\sqrt{\int_{\sigma}^0\int_{B_1}|y\cdot\nabla w|^2\,dy\,ds}\leq C\sqrt{|\sigma|}
\end{equation*}
by Step 2 and \eqref{C17}.
\begin{multline*}
 \left|\int_{\sigma}^0\int_{B_1} \frac{\partial_s w\,w|y|^2}{(1-|y|^2)^{3/2}}dy\,ds\right|\\ \leq \sqrt{\int_{\sigma}^0\int_{B_1} \frac{(\partial_sw)^2}{(1-|y|^2)^{3/2}}\,dy\,ds}\sqrt{\int_{\sigma}^0\int_{B_1} \frac{w^2}{(1-|y|^2)^{3/2}}\,dy\,ds}\leq C\sqrt{|\sigma|}
\end{multline*}
by Step 2 and \eqref{C18}. This concludes Step 3.

\EMPH{Step 4: conclusion of the proof} We first prove by contradiction that $\tE(-1)<E(\vu_0)$. If not, we see from the identity \eqref{C21} and Lemma \ref{L:C5} that $\int|\partial_sw(s,y)|^2\,dy=0$ for almost every $s\in (-1,0)$. Thus $w$ is independent of $s\in(-1,0)$, and since $\supp w\subset\{|y|\leq 1-e^s\}$, $w=0$ for $s\in (-1,0)$. Thus $u\equiv 0$, contradicting the assumption $T_+(u)=1$. By Step 3, we obtain, for $\sigma\leq -2$,
$$\int_{\sigma}^{-1} \left(\int_{B_1} \frac{|w(s,y)|^{\frac{2N}{N-2}}}{N(1-|y|^2)^{1/2}}\,dy-\tE(-1)\right)\,ds \leq C|\sigma|^{1/2}.$$
Going back to the variables $(t,x)$, we deduce, for $\sigma \leq -2$,
$$\int_{2-e^{-\sigma}}^{2-e} \left(\int_{\RR^N} \frac{|u(t,x)|^{\frac{2N}{N-2}}}{N\left(1-\frac{|x|^2}{(2-t)^2}\right)^{1/2}}\,dx-\tE(-1)\right)\,\frac{dt}{2-t} \leq C|\sigma|^{1/2}.$$
Dropping the factor 
$\frac{1}{\left(1-\frac{|x|^2}{(2-t)^2}\right)^{1/2}}\geq 1$, we obtain
$$\int_{T}^{2-e} \left(\int_{\RR^N} \frac{1}{N}|u(t,x)|^{\frac{2N}{N-2}}\,dx-\tE(-1)\right)\,\frac{dt}{2-t} \leq C|\log(2-T)|^{1/2},$$
where $T=2-e^{-\sigma}\leq 2-e^2$. Dividing by $\frac{1}{\log(2-T)}$ and using that $\tE(-1)<E(\vu_0)$ we obtain the desired conclusion \eqref{C5}.
\end{proof}

\begin{proof}[Proof of Proposition \ref{P:lambda}]
The proof is close to the end of Section 6 of \cite{KeMe08}. Using the same self-similar change of variables as in the preceding proof, we construct, assuming that the conclusion of Proposition \ref{P:lambda} is not true, a nonzero solution to a singular elliptic equation, yielding a contradiction with a result of \cite{KeMe08}.

We can assume, without loss of generality, that $T_+(u)=1$ and thus $|x|\leq 1-t$ on the support of $u$. By finite speed of propagation
\begin{equation*}
 \limsup_{t\to-\infty}\frac{\lambda(t)}{|t|}<\infty.
\end{equation*} 
Furthermore, since $P(\vu_0)=0$, we have
\begin{equation}
 \label{L1}
\lim_{t\to -\infty}\frac{x(t)}{t}=0.
\end{equation} 
(see \cite{KeMe08} for the detailed proofs of these properties).

We argue by contradiction, assuming that there exists a sequence of times $\tau_n\to -\infty$ such that
\begin{equation}
 \label{L2}
\lim_{n\to \infty}\frac{\lambda(\tau_n)}{|\tau_n|}=\ell \in (0,+\infty).
\end{equation} 
The solution $u$ satisfies the assumptions of Proposition \ref{P:critical}.  We introduce as above the self-similar variables $y$ and $s$ (see \eqref{C13}) and define $w$ by \eqref{C14}.

\EMPH{Step 1: compactness} Let $\sigma_n=-\log(2-\tau_n)$. Let 
$$w_n(s)=w(\sigma_n+s),\quad s<-\sigma_n.$$
In this step we prove that there exists (after extraction of a subsequence) a small $s_0>0$ and $w_*\in C^0([0,s_0],\dot{H}^1)$ such that $\partial_sw_*\in C^0([0,s_0],L^2)$ and 
\begin{equation}
 \label{L3}
\lim_{n\to\infty}\sup_{0\leq s\leq s_0}\big\|\big(w_n(s)-w_*(s),\partial_sw_n(s)-\partial_sw_*(s)\big)\big\|_{\dot{H}^1\times L^2}=0.
\end{equation} 
Indeed, let 
\begin{align*}
v_n(\tau,z)&=(2-\tau_n)^{\frac{N-2}{2}}u\big(\tau_n+(2-\tau_n)\tau, (2-\tau_n)z\big)\\
u_n(\tau,z)&=\lambda(\tau_n)^{\frac{N-2}{2}}u\big(\tau_n+\lambda(\tau_n)\tau,\lambda(\tau_n)z+x(\tau_n)\big).
\end{align*} 
By the precompactness of $K$, $(u_n(0),\partial_{\tau}u_n(0))$ has (after extraction of a subsequence) a limit in $\dot{H}^1\times L^2$ as $n$ goes to infinity. Noting that 
$$v_n(\tau,z)=\left( \frac{2-\tau_n}{\lambda(\tau_n)} \right)^{\frac{N-2}{2}} u_n\left( \frac{2-\tau_n}{\lambda(\tau_n)}\tau,\frac{2-\tau_n}{\lambda(\tau_n)}z-\frac{x(\tau_n)}{\lambda(\tau_n)} \right),$$
and combining with \eqref{L1} and \eqref{L2}, we see that $(v_n(0),\partial_{\tau}v_n(0))$ has a limit in $\dot{H}^1\times L^2$ as $n$ goes to infinity. We denote by $(v_0,v_1)$ this limit, and by $v_*$ the solution of \eqref{CP} with initial data $(v_0,v_1)$ at $t=0$. 

Fix $\tau_0\in [0,T_+(v))$, and let $s_0$ such that $s_0=-\log(1-\tau_0)$. By standard perturbation theory for equation \eqref{CP},
\begin{equation}
 \label{L5}
\lim_{n\to\infty}\sup_{0\leq \tau\leq \tau_0}\big\|\big(v_n(\tau)-v_*(\tau),\partial_{\tau}v_n(\tau)-\partial_{\tau}v_*(\tau)\big)\big\|_{\dot{H}^1\times L^2}=0.
\end{equation} 
Next, notice that 
$$w_n(s,y)=e^{-\frac{N-2}{2}s} v_n(1-e^{-s},e^{-s}y),$$
and thus, by \eqref{L5},
\begin{equation}
 \label{L5'}
\lim_{n\to\infty}\sup_{0\leq s\leq s_0}\big\|w_n(s)-w_*(s)\big\|_{\dot{H}^1}=0,
\end{equation} 
where
\begin{equation}
 \label{defw*}
w_*(s,y)=e^{-\frac{N-2}{2}s} v_*(1-e^{-s},e^{-s}y).
\end{equation} 
By \eqref{C20'}, 
$
\partial_sw_n=-\frac{N-2}{2}w_n-y\cdot\nabla w_n+e^{-\frac N2(s+\sigma_n)}\partial_tu(2-e^{-(\sigma_n+s)},e^{-(\sigma_n+s)}y)$, and thus
\begin{equation}
\label{L6}
\partial_sw_n=-\frac{N-2}{2}w_n-y\cdot\nabla w_n+e^{-\frac N2s}\partial_{\tau}v_n(1-e^{-s},e^{-s}y).
 \end{equation}
Since $|y|\leq 1$ on the support of $w_n$ (see \eqref{C15}), we deduce from \eqref{L5'}
\begin{equation*}
\lim_{n\to\infty}\left(\sup_{0\leq s\leq s_0}\big\|w_n(s)-w_*(s)\big\|_{L^2}+\sup_{0\leq s\leq s_0}\big\|y\cdot\nabla w_n(s)-y\cdot \nabla w_*(s)\big\|_{L^2}\right)=0.
\end{equation*}
In view of \eqref{L5}, \eqref{L6}, and the equality 
$$\partial_sw_*(s,y)=-\frac{N-2}{2}w_*-y\cdot \nabla w_*+e^{-\frac{N}{2}s}\partial_{\tau} v_*(1-e^{-s},e^{-s}y),$$
we obtain
$$ \lim_{n\to\infty}\sup_{0\leq s\leq s_0}\big\|\partial_s w_n(s)-\partial_sw_*(s)\big\|_{\dot{H}^1}=0,$$
which concludes Step 1.

\EMPH{Step 2: elliptic equation}
We prove that $w_*$ is independent of $s$, not identically $0$, and satisfies
\begin{gather}
\label{L7}
\supp w_*\subset B_1\\
\label{L8}
\frac{1}{\rho} \Div\left(\rho\nabla w_*-\rho(y\cdot\nabla w_*)y\right) -\frac{N(N-2)}{4}w_*+|w_*|^{\frac{4}{N-2}}w_*=0\\
\label{L8'}
\int_{B_1}\frac{|w_*|^2}{(1-|y|^2)^2}\,dy<\infty\\
\label{L9}
\int_{B_1} \frac{|w_*|^{\frac{2N}{N-2}}}{(1-|y|^2)^{1/2}} \,dy<\infty\\
\label{L10}
\int \frac{|\nabla w_*|^2-(y\cdot\nabla w_*)^2}{(1-|y|^2)^{1/2}}\,dy<\infty.
\end{gather}
This contradicts Proposition 6.10 in \cite{KeMe08}, concluding the proof of Proposition \ref{P:lambda}.

The condition \eqref{L7} of the support follows immediately from \eqref{L3} and the corresponding condition \eqref{C15} on the support of $w$. By \eqref{C30}, 
$$\lim_{n\to\infty} \int_0^{s_0}\int_{B_1} (\partial_sw_n)^2\,dy\,ds=0,$$
which proves, combining with \eqref{L3}, that $\partial_s w_*$ is almost everywhere zero, and thus that $w_*$ is independent of $s$.  

Assume that $w_*\equiv 0$. Then $v_*\equiv 0$, and the small data theory for \eqref{CP}, together with \eqref{L5} implies that $u$ is global, a contradiction. Thus $w_*$ is not identically $0$.

The bound \eqref{L8'} follows from \eqref{L7}, Hardy's inequality and the fact that $w_*$ is in $\hdot(\RR^N)$. 

We next prove \eqref{L9}. Using the identity \eqref{C22} with $s_1=\sigma_n$ and $s_2=\sigma_n+s_0$ we get, combining with \eqref{C17}, \eqref{C18}, Lemma \ref{L:C6} and the fact that $\tilde{E}$ is nondecreasing
\begin{equation}
\label{L11}
\limsup_{n\to\infty} \int_{0}^{s_0} \int_{B_1} \frac{|w_n|^{\frac{2N}{N-2}}}{(1-|y|^2)^{1/2}} \,dy\,ds=M<\infty 
\end{equation} 
(see Step 1 of the proof of Proposition \ref{P:critical}, p. \pageref{C28}, for similar arguments). By \eqref{L3} and the Sobolev inequality, and since $w_*$ is independent of time we deduce that for all $\eps>0$,
$$ \int_{|y|\leq 1-\eps} \frac{|w_*|^{\frac{2N}{N-2}}}{(1-|y|^2)^{1/2}} \,dy\leq M/s_0,$$
and \eqref{L9} follows.

It remains to prove \eqref{L10}. By the definition of $\tE$,
\begin{multline*}
\frac{1}{2}\int_{0}^{s_0} \int_{B_1}\frac{|\nabla w_n(s,y)|^2-(y\cdot\nabla w_n(s,y))^2}{(1-|y|^2)^{1/2}}\,dy\,ds\\
\leq \int_{\sigma_n}^{\sigma_n+s_0} \tE(s)\,ds+\frac{N-2}{2N}\int_{0}^{s_0} \int_{B_1}\frac{|w_n(s,y)|^{\frac{2N}{N-2}}}{(1-|y|^2)^{1/2}} \,dy\,ds.
\end{multline*}
Combining with \eqref{L11}, Lemma \ref{L:C5} and the monotonicity of $\tE$, we obtain
\begin{equation*}
\limsup_{n\to\infty} \int_{0}^{s_0} \int_{B_1}\frac{|\nabla w_n(s,y)|^2-(y\cdot\nabla w_n(s,y))^2}{(1-|y|^2)^{1/2}}\,dy\,ds\leq s_0 E(\vu_0)+\frac{N-2}{2N}M. 
\end{equation*} 
which yields \eqref{L10} with a similar argument as before.
\end{proof}

\section{Energy-supercritical wave equation}
\label{S:super}
In this section, we let $N\geq 3$, $p>\frac{N+2}{N-2}$ and consider the supercritical focusing wave equation:
\begin{equation}
 \label{CPsup}
 \left\{\begin{aligned}
         \partial_t^2u -\Delta u&=|u|^{p-1}u,\quad t\in I,\; x\in \RR^N\\
         (u,\partial_tu)_{\restriction t=0}&=\vec{u}_0\in \dot{H}^{s_c}\times \dot{H}^{s_c-1} ,
        \end{aligned}\right.
\end{equation} 
where $s_c=\frac{N}{2}-\frac{2}{p-1}>1$, $I$ is an interval containing $0$, and the unknown function $u$  is again real-valued. 

We assume furthermore that $p$ is an odd integer, or that $p$ is large enough ($p>N/2$ is sufficient), so that 
the equation \eqref{CPsup} is locally well-posed: for any $\vec{u}_0\in \dot{H}^{s_c}\times\dot{H}^{s_c-1}$, there exists an unique solution $\vec{u}=(u,\partial_tu)\in C^0\left(I_{\max}(\vec{u}_0),\dot{H}^{s_c}\times \dot{H}^{s_c-1} \right)$ defined on a maximal interval of existence $I_{\max}(u)=(T_-(\vec{u}_0),T_+(\vec{u}_0))$ and that satisfies \eqref{CPsup} in the Duhamel sense.

We say that a solution $u$ of \eqref{CPsup} has the compactness property when there exist $\lambda(t)>0$, $x(t)\in \RR^N$, defined for $t\in I_{\max}(\vec{u}_0)$ and such that
\begin{equation}
 \label{defKsup}
K=\left\{\left( \lambda(t)^{\frac{2}{p-1}}u\left( t,\lambda(t)y+x(t)\right), \lambda(t)^{\frac{2}{p-1}-1}\partial_tu\left( t,\lambda(t)y+x(t) \right)\right)\;;\; t\in I_{\max}(\vec{u}_0)\right\}
 \end{equation} 
has compact closure in $\dot{H}^{s_c}\times \dot{H}^{s_c-1} $.  In this section we prove:
\begin{prop}
\label{P:mainsup}
 Let $p$ be as above, and $u$  a solution of \eqref{CPsup} with the compactness property. Then $u$ is global.
\end{prop}
We conjecture that when $p>\frac{N+2}{N-2}$, the only solution of \eqref{CPsup} with the compactness property is $0$. This would imply, in particular, that any solution of \eqref{CPsup} which is bounded in the critical space $\dot{H}^{s_c}\times \dot{H}^{s_c-1}$ scatters. This conjecture was settled in space dimension $3$ for radial solutions in \cite{DuKeMe12c}. We refer to \cite{KeMe11,KillipVisan11} for the corresponding defocusing equation.

The proof relies on classical monotonicity formulas. Note that Proposition \ref{P:mainsup} excludes in particular self-similar blow-up, generalizing \cite[Proposition 2.2]{DuKeMe12c} with a simpler proof. 
\begin{proof}
We let $u$ be a solution of \eqref{CPsup} with the compactness property, $I_{\max}(\vu_0)=(T_-,T_+)$ the maximal interval of existence of $u$, and $\lambda(t)$, $x(t)$, $t\in (T_-,T_+)$ such that $K$ defined by \eqref{defKsup} has compact closure in $\dot{H}^{s_c}\times \dot{H}^{s_c-1} $. We argue by contradiction, assuming that $T_-$ is finite.

\EMPH{Step 1. Condition on the support}

We prove that there exists $x_-\in \RR^N$ such that 
\begin{equation}
 \label{S2}
 \supp u\subset\Big\{ |x-x_-|\leq |T_--t|\Big\}.
\end{equation} 
The proof is quite standard. We give it for the sake of completeness.
By the local Cauchy theory for \eqref{CPsup}, 
$$\lim_{t\overset{>}{\longrightarrow} T_-}\lambda(t)=0.$$
By finite speed of propagation, $x(t)$ is bounded on $(T_-,T_+)$. Let $\{\tau_n\}_n$ be a sequence of times in $(T_-,T_+)$ such that $\{\tau_n\}_n$ goes to $T_-$, and $x(\tau_n)$ has a limit $x_-\in \RR^N$ as $n$ goes to infinity. 

We fix $t\in (T_-,0]$. Let $\eps$ be a small positive number. Let $\chi\in C_0^{\infty}(\RR^N)$ such that $\chi(x)=1$ if $|x|\geq 1$ and $\chi(x)=0$ if $|x|\leq \frac 12$. By the precompactness of $K$, we can find $R>0$ such that 
\begin{equation}
 \label{small}
 \forall n,\quad \left\|\left(u(\tau_n)\chi\left( \frac{\cdot-x(\tau_n)}{\lambda(\tau_n)R} \right),\partial_t u(\tau_n)\chi\left( \frac{\cdot-x(\tau_n)}{\lambda(\tau_n)R} \right)\right)\right\|_{\dot{H}^{s_c}\times \dot{H}^{s_c-1}}<\eps.
\end{equation}
Let $\tilde{u}_n$ be the solution of \eqref{CPsup} with initial data 
$$\left(u(\tau_n)\chi\left( \frac{\cdot-x(\tau_n)}{\lambda(\tau_n)R} \right),\partial_t u(\tau_n)\chi\left( \frac{\cdot-x(\tau_n)}{\lambda(\tau_n)R} \right)\right).$$
By the small data theory, this solution is global and  
$$ \forall \tau,\quad \left\|(\tilde{u}_n(\tau),\partial_t\tilde{u}_n(\tau))\right\|_{\dot{H}^{s_c}\times \dot{H}^{s_c-1}}\leq 2\eps.$$
Combining with finite speed of propagation and Sobolev inequality, we obtain:
$$ \|u(t)\|_{L^{p_c}\left(\{|x-x(\tau_n)|\geq 2\lambda(\tau_n)R+|t-\tau_n|\right)}\leq C\eps,$$
where $p_c=\frac{(p-1)N}{2}$.
Letting $n\to \infty$, we obtain
$$\|u(t)\|_{L^{p_c}\left(\{|x-x_-|\geq |t-T_-|\}\right)}\leq C\eps.$$
Since $\eps$ is arbitrary, we deduce \eqref{S2}.

\EMPH{Step 2. Monotonicity formula and end of the proof}

By Step 1, $\vu(t)\in \dot{H}^{1}\times L^2$ for all $t\in (T_-,T_+)$ and, letting 
\begin{equation}
 \label{S3}
 y(t)=\int_{\RR^N}u^2,
\end{equation} 
we obtain with equation \eqref{CPsup} that $y$ is twice differentiable on $(T_-,T_+)$ and 
\begin{align}
 \label{S4}
 y'(t)&=2\int_{\RR^N} u\partial_tu\\
 \label{S5}
 y''(t)&=2\int_{\RR^N}(\partial_tu)^2-2\int |\nabla u|^2+2\int |u|^{p+1}.
\end{align}
Furthermore, the energy $E(\vu(t))$ defined in \eqref{defE} is well-defined and conserved with the flow. Since $u$ is bounded in $\dot{H}^{s_c}\times \dot{H}^{s_c-1} $, the condition \eqref{S2} on the support of $u$ implies
\begin{equation}
 \label{S6}
 \lim_{t\overset{>}{\longrightarrow} T_-}E(\vu(t))=0\quad \text{and}\quad \lim_{t\overset{>}{\longrightarrow} T_-}y(t)=\lim_{t\overset{>}{\longrightarrow} T_-}y'(t)=0.
\end{equation} 
By the conservation of the energy
\begin{equation}
 \label{S7}
 \forall t\in (T_-,T_+), \quad E(\vu(t))=0,
\end{equation} 
and we can rewrite \eqref{S5} as
\begin{equation}
 \label{S8} 
 y''(t)=(p+3)\int(\partial_tu)^2+(p-1)\int |\nabla u|^2>0.
\end{equation} 
Combining with the limit of $y'(t)$ in \eqref{S6} we deduce
\begin{equation}
 \label{S9}
 \forall t\in (T_-,T_+),\quad y'(t)>0.
\end{equation} 
We first exclude the case $T_+<\infty$. In this case, similar arguments than above yield
\begin{equation}
 \label{S10}
\lim_{t\overset{<}{\longrightarrow} T_+}y'(t)=0,
\end{equation} 
contradicting the strict convexity \eqref{S8} of $y$ and the fact that $\lim_{t\to T_-}y'(t)=0$.

Thus we must have $T_+=+\infty$. By \eqref{S4} and \eqref{S8},
\begin{equation}
 \label{S11}
 y'(t)^2\leq \frac{4}{p+3} y''(t)y'(t).
\end{equation} 
Since $y'(t)>0$ for all $t>T_-$, it is straightforward that $y^{-(p-1)/4}$ is strictly decreasing and strictly concave, contradicting $T_+=+\infty$.
\end{proof}
\section{Nonlinear Schr\"odinger equation}
\label{S:NLS}
In this section we consider the energy-supercritical nonlinear Schr\"odinger equation
\begin{equation}
 \label{NLS}
 \left\{
 \begin{aligned}
 i\partial_tu+\Delta u+\iota |u|^{p-1}u&=0,\quad t\in I,\; x\in \RR^N\\
 u_{\restriction t=0}&=u_0\in \dot{H}^{s_c}(\RR^N),
 \end{aligned}
 \right.
\end{equation} 
where $I$ is a real interval containing $0$. Here $N\geq 3$, $p>\frac{N+2}{N-2}$, $\iota\in \{\pm 1\}$ and $s_c=\frac{N}{2}-\frac{2}{p-1}>1$. We assume again that $p$ is an odd integer, or that $p$ is large enough (say $p>N/2$), so that the equation is locally well-posed in $\dot{H}^{s_c}$. A solution $u$ of \eqref{NLS} \emph{with the compactness property} is by definition a solution with maximal interval of existence $(T_-,T_+)$ such that there exists $\lambda(t)>0$, $x(t)\in \RR^N$, defined for $t\in (T_-,T_+)$, such that
\begin{equation}
\label{defK2}
K=\left\{ \lambda(t)^{\frac{2}{p-1}} u\big( t,\lambda(t)y+x(t)\big),\; t\in (T_-,T_+) \right\} 
\end{equation} 
has compact closure in $\dot{H}^{s_c}$. In this section we prove:
\begin{prop}
\label{P:NLS}
 Let $p$ be as above, and $u$ a solution of \eqref{NLS} with the compactness property. Then $u$ is global.
\end{prop}

The proof of Proposition \ref{P:NLS} is based on differentiation of the localized $L^2$ norm as in the energy-critical case $p=\frac{N+2}{N-2}$  (see case 1 in the proof of Proposition 5.3 in \cite{KeMe06}), with an additional iteration of the argument (see step 2 below). 

As in the case of the wave equation, we conjecture that the only solution of \eqref{NLS} with the compactness property with $p>\frac{N+2}{N-2}$ is $0$. 
This was proved in the defocusing case $\iota=-1$ in dimension $N\geq 5$ in \cite{KillipVisan10b}. The proof in \cite{KillipVisan10b} that a solution with the compactness property is global, relying on energy conservation, is specific to the defocusing case.

We will need the following claim:
\begin{claim}
\label{C:embedding}
 There exists a constant $C>0$ such that for all $f\in \dot{H}^{s_c}(\RR^N)$, for all $R>0$,
 \begin{gather*}
 \int_{|x|\leq R}|f(x)|^2\,dx\leq CR^{2s_c}\|f\|^2_{\dot{H}^{s_c}},\quad \int_{|x|\leq R}|\nabla f(x)|^2\,dx\leq C\|f\|^{\frac{2}{s_c}}_{\dot{H}^{s_c}}\left(\int_{|x|\leq 2R} |f(x)|^2\,dx\right)^{1-\frac{1}{s_c}}.
 \end{gather*} 
\end{claim}
\begin{proof}
 It is sufficient to prove both inequalities for $R=1$. The general case follows by scaling. 
 
 The first inequality with $R=1$ is elementary. The second one is an immediate consequence of the interpolation inequality
 $$\|g\|_{\dot{H}^1}\leq C\|g\|_{\dot{H}^{s_c}}^{\frac{1}{s_c}}\|g\|_{L^2}^{1-\frac{1}{s_c}}$$
 applied to $g=\varphi f$, where $\varphi\in C_0^{\infty}(\RR^N)$, $\varphi(x)=1$ if $|x|\leq 1$ and $\varphi(x)=0$ if $|x|\geq 2$.
\end{proof}
\begin{proof}[Proof of Proposition \ref{P:NLS}]
 We argue by contradiction. Let $u$ be a solution of \eqref{NLS} with the compactness property, and $\lambda(t)$, $x(t)$ such that $K$ defined by \eqref{defK2} has compact closure in $\dot{H}^{s_c}$. Assume that the maximal forward time of existence $T_+$ of $u$ is finite. 
 
 Let $\varphi\in C_0^{\infty}(\RR^N)$ be a nonnegative function such that $\varphi(x)=1$ if $|x|\leq 1$ and $\varphi(x)=0$ if $|x|\geq 2$. Let, for $R>0$,
 \begin{equation}
  \label{N1}
  V_R(t)=\int |u(t,x)|^2\varphi\left( \frac{x}{R} \right)\,dx.
 \end{equation} 
 Then $V_R$ is differentiable and 
 \begin{equation}
  \label{N2}
  V_R'(t)=\frac{2}{R}\im \int \nabla \varphi\left( \frac{x}{R} \right)\overline{u}(t,x)\cdot\nabla u(t,x)\,dx.
 \end{equation} 
 \EMPH{Step 1} We prove that for all $R>0$
 \begin{equation}
  \label{N3}
  \lim_{t\overset{<}{\longrightarrow}T_+}V_R(t)=0.
 \end{equation} 
 First note that by the local Cauchy theory for \eqref{NLS} and the compactness of the closure of $K$ in $\dot{H}^{s_c}$, one has 
 \begin{equation}
  \label{lim_lambda}
  \lim_{t\overset{<}{\longrightarrow}T_+}\lambda(t)=0.
 \end{equation} 
 Let 
 $$ v(t,\cdot)=\lambda(t)^{\frac{2}{p-1}}u\left( t,\lambda(t)\cdot+x(t)\right)\in K.$$
Then
\begin{multline*}
V_R(t)=\lambda(t)^{2s_c} \int \varphi\left(\frac{\lambda(t)y+x(t)}{R}\right)|v(t,y)|^2\,dy= \\
=\underbrace{\lambda(t)^{2s_c} \int_{|y|\leq \eps R/\lambda(t)}\ldots}_{A_{\eps}(t)}+\underbrace{
\lambda(t)^{2s_c} \int_{|y|\geq \eps R/ \lambda(t)}\ldots}_{B_{\eps}(t)}
\end{multline*}
By H\"older's inequality and Sobolev's embedding, using that $v$ is bounded in $\dot{H}^{s_c}$,
\begin{equation*}
A_{\eps}(t) \leq C\lambda(t)^{2s_c} \left(\int_{|y|\leq \eps R/\lambda(t)} \,dy\right)^{\frac{p_c-2}{p_c}}\|v(t)\|_{L^{p_c}}^2\leq C\lambda(t)^{2s_c}\left( \int_{|y|\leq \eps R/\lambda(t)}dy \right)^{\frac{p_c-2}{p_c}}\leq C(\eps R)^{2s_c}
\end{equation*}
where $p_c=\frac{N(p-1)}{2}$, so that $\dot{H}^{s_c}$ is embedded into $L^{p_c}$ . Thus $A_{\eps}(t)$ is small (uniformly in $t$) when $\eps$ is small.

Using H\"older's inequality again, we obtain
\begin{multline*}
B_{\eps}(t)\leq \lambda^{2s_c}\left(\int \Big|\varphi\left( \frac{\lambda(t)y+x(t)}{R} \right)\Big|^{\frac{p_c}{p_c-2}}\,dy\right)^{\frac{p_c-2}{p_c}}\left( \int_{|y|\geq R\eps/\lambda(t)}|v(t,y)|^{p_c}\,dy \right)^{\frac 2p_c}\\
\leq \left(\int |\varphi(x/R)|^\frac{p_c}{p_c-2}\,dx\right)^{\frac{p_c-2}{p_c}}\left(\int_{|y|\geq R\eps/\lambda(t)}|v(t,y)|^{p_c}\,dy\right)^{\frac{2}{p_c}}.
\end{multline*}
Using \eqref{lim_lambda} and the fact that $v(t)$ stays in a compact subset of $\dot{H}^{s_c}$, we obtain $\lim_{t\overset{<}{\longrightarrow}T_+}B_{\eps}(t)=0$, and \eqref{N3} follows.

\EMPH{Step 2} We conclude the proof, showing that there exists $\beta<0$ such that
\begin{equation}
 \label{N4}
 \forall t\in [0,T_+), \; \forall R>0,\quad |V_R(t)|\leq CR^{\beta}.
\end{equation} 
Indeed, if \eqref{N4} holds, letting $R\to \infty$ at time $t=0$, we obtain $u_0=0$, a contradiction. 

We prove \eqref{N4} as a consequence of the following implication (for $\alpha\in \RR$).
\begin{multline}
 \label{N5}
 \Big(\exists C>0,\; \forall t\in [0,T_+),\; \forall R>0,\; |V_R(t)|\leq CR^{\alpha}\Big)\\
 \Longrightarrow \Big(\exists C>0,\; \forall t\in [0,T_+),\; \forall R>0,\; |V_R(t)|\leq CR^{\alpha-1-\frac{\alpha}{2s_c}}\Big).
\end{multline}
By Claim \ref{C:embedding} the first line of \eqref{N5} holds with $\alpha=2s_c$. Thus \eqref{N5} implies \eqref{N4}. 

To prove \eqref{N5}, notice that by \eqref{N2}, Cauchy-Schwarz inequality and Claim \ref{C:embedding},
\begin{multline*}
|V'_R(t)|\leq \frac{C}{R}\left(\int_{|x|\leq 2R}|u(t,x)|^2dx\right)^{\frac{1}{2}}\left(\int_{|x|\leq 2R}|\nabla u(t,x)|^2dx\right)^{\frac{1}{2}}\\
\leq \frac{C}{R}\left(\int_{|x|\leq 4R}|u(t,x)|^2\,dx\right)^{1-\frac{1}{2s_c}}\|u(t)\|^{\frac{1}{s_c}}_{\dot{H}^{s_c}}.
\end{multline*} 
Using that $u$ is bounded in $\dot{H}^{s_c}$, we deduce
$$|V'_R(t)|\leq \frac{C}{R}V_{4R}(t)^{1-\frac{1}{2s_c}}.$$
Integrating between $t$ and $T_+$ and using Step 1, we obtain
$$|V_R(t)|\leq \frac{C}{R}\int_{t}^{T_+}V_{4R}(\tau)^{1-\frac{1}{2s_c}}\,d\tau,$$
which implies \eqref{N5}, concluding the proof.
 \end{proof}

\appendix
\section{Proof of the Claim}
Let $H(t)=\frac{1}{2+t}\int_0^tg(s)ds$. By \eqref{C10}, and since $g$ is bounded,
\begin{equation}
 \label{C12}
 H(t)=\frac{1}{t}\int_0^tg(s)\,ds+O\left(\frac{1}{t}\right)\underset{t\to+\infty}{\longrightarrow} \ell.
\end{equation} 
Furthermore, integrating by parts and using that $g((t)=\frac{d}{dt}\left( (2+t)H(t) \right)$, we obtain
$$\frac{1}{\log(2+T)}\int_0^T\frac{g(t)}{2+t}\,dt=\frac{1}{\log(2+T)}\left(\int_0^T H(t)\,dt+H(T)-H(0)\right).$$
By \eqref{C12}, $\lim_{T\to+\infty}\frac{1}{\log(2+T)}\left(H(T)-H(0)\right)=0$. Moreover, by the change of variable $s=\log(2+t)$,
$$\frac{1}{\log(2+T)}\int_0^T\frac{H(t)}{2+t}\,dt=\frac{1}{\log(2+T)}\int_0^{\log(2+T)}H\left( e^s-2 \right)\,ds.$$
By Ces\`aro mean, we deduce from \eqref{C12} that the preceding goes to $\ell$ as $T\to\infty$ and thus
$$\lim_{T\to\infty} \frac{1}{\log(2+T)}\int_0^T \frac{H(t)}{2+t}\,dt=\ell,$$
which concludes the proof.


\begin{thebibliography}{10}

\bibitem{BrMa97}
{\sc Brezis, H., and Marcus, M.}
\newblock Hardy's inequalities revisited.
\newblock {\em Ann. Scuola Norm. Sup. Pisa Cl. Sci. (4) 25}, 1-2 (1997),
  217--237 (1998).
\newblock Dedicated to Ennio De Giorgi.

\bibitem{DuKeMe11a}
{\sc Duyckaerts, T., Kenig, C., and Merle, F.}
\newblock Universality of blow-up profile for small radial type {II} blow-up
  solutions of the energy-critical wave equation.
\newblock {\em J. Eur. Math. Soc. (JEMS) 13}, 3 (2011), 533--599.

\bibitem{DuKeMe12c}
{\sc Duyckaerts, T., Kenig, C., and Merle, F.}
\newblock Scattering for radial, bounded solutions of focusing supercritical
  wave equations.
\newblock {\em IMRN\/} (2012).

\bibitem{DuKeMe12}
{\sc Duyckaerts, T., Kenig, C., and Merle, F.}
\newblock Universality of the blow-up profile for small type {II} blow-up
  solutions of the energy-critical wave equation: the nonradial case.
\newblock {\em J. Eur. Math. Soc. (JEMS) 14}, 5 (2012), 1389--1454.

\bibitem{DuKeMe13Pb}
{\sc Duyckaerts, T., Kenig, C., and Merle, F.}
\newblock Solutions of the focusing, energy-critical wave equation with the
  compactness property.
\newblock Preprint arXiv:1402.0365, 2014.

\bibitem{DuKeMe15a}
{\sc Duyckaerts, T., Kenig, C., and Merle, F.}
\newblock Profiles for bounded solutions of dispersive equations, with
  applications to energy-critical wave and {S}chr\"odinger equations.
\newblock {\em Commun. Pure Appl. Anal. 14}, 4 (2015), 1275--1326.

\bibitem{GiKo89}
{\sc Giga, Y., and Kohn, R.~V.}
\newblock Nondegeneracy of blowup for semilinear heat equations.
\newblock {\em Comm. Pure Appl. Math. 42}, 6 (1989), 845--884.

\bibitem{KeMe06}
{\sc Kenig, C.~E., and Merle, F.}
\newblock Global well-posedness, scattering and blow-up for the
  energy-critical, focusing, non-linear {S}chr\"odinger equation in the radial
  case.
\newblock {\em Invent. Math. 166}, 3 (2006), 645--675.

\bibitem{KeMe08}
{\sc Kenig, C.~E., and Merle, F.}
\newblock Global well-posedness, scattering and blow-up for the energy-critical
  focusing non-linear wave equation.
\newblock {\em Acta Math. 201}, 2 (2008), 147--212.

\bibitem{KeMe11}
{\sc Kenig, C.~E., and Merle, F.}
\newblock Nondispersive radial solutions to energy supercritical non-linear
  wave equations, with applications.
\newblock {\em Amer. J. Math. 133}, 4 (2011), 1029--1065.

\bibitem{KillipVisan10b}
{\sc Killip, R., and Visan, M.}
\newblock Energy-supercritical {NLS}: Critical $\dot{H}^s$-bounds imply
  scattering.
\newblock {\em Communications in Partial Differential Equations 35}, 6 (2010),
  945--987.

\bibitem{KillipVisan11}
{\sc Killip, R., and Visan, M.}
\newblock The defocusing energy-supercritical nonlinear wave equation in three
  space dimensions.
\newblock {\em Trans. Amer. Math. Soc. 363}, 7 (2011), 3893--3934.

\bibitem{MeZa03}
{\sc Merle, F., and Zaag, H.}
\newblock Determination of the blow-up rate for the semilinear wave equation.
\newblock {\em Amer. J. Math. 125}, 5 (2003), 1147--1164.

\bibitem{Tao07DPDE}
{\sc Tao, T.}
\newblock A (concentration-) compact attractor for high-dimensional non-linear
  {S}chr{\"o}dinger equations.
\newblock {\em Dynamics of PDE 4}, 1 (2007), 1--53.

\end{thebibliography}
\end{document}